\documentclass{article}
\usepackage{graphicx}
\usepackage{amsfonts}
\usepackage{amsmath}
\usepackage{amssymb}
\usepackage{amsthm}

\usepackage[numbers]{natbib}
\usepackage[english]{babel}
\newtheorem{theorem}{Theorem}
\usepackage{color}

\usepackage[english]{babel}

\usepackage[letterpaper,top=2cm,bottom=2cm,left=3cm,right=3cm,marginparwidth=1.75cm]{geometry}

\usepackage{amsmath}
\usepackage{graphicx}
\usepackage{amsthm,amsmath,multirow,graphicx,subfigure,booktabs,url,mathtools,wrapfig,mathrsfs,bm,relsize,caption,color, enumitem, fullpage,float,bbm}
\usepackage[colorlinks=true, allcolors=blue]{hyperref}
\usepackage{amsthm,amsmath,multirow,graphicx,subfigure,booktabs,url,mathtools,wrapfig,mathrsfs,bm,relsize,caption,color, enumitem, fullpage,float,bbm}

\numberwithin{equation}{section}
\newtheorem{conj}{Conjecture}

\newtheorem{lemma}[conj]{Lemma}

\newtheorem{definition}{Definition}

\newtheorem{innercustomgeneric}{\customgenericname}
\providecommand{\customgenericname}{}
\newcommand{\newcustomtheorem}[2]{%
	\newenvironment{#1}[1] 
	{%
		\renewcommand\customgenericname{#2}%
		\renewcommand\theinnercustomgeneric{##1}%
		\innercustomgeneric
	}
	{\endinnercustomgeneric}
}

\newcustomtheorem{customAss}{Assumption}

\theoremstyle{remark}

\def\0{\bm 0}
\def\1{\mathbbm 1}
\def\b1{\bm 1}

\DeclareMathOperator*{\argmin}{arg\,min}
\DeclareMathOperator*{\argmax}{arg\,max}

\title{A scalar matching factor on the Birkhoff polytope characterizing permutation and uniform matrices}
\author{Suvadip Sana\footnote{Email : ss2776@cornell.edu} \\ Department of Statistics and Data Science, Cornell University}
\date{}
\begin{document}
\maketitle
\begin{abstract}
    Birkhoff polytope is the set of all bistochastic matrices (also known as doubly stochastic matrices). Bistochastic matrices form a special class of stochastic matrices where each row and column sums up to one. Permutation matrices and uniform matrices are special extreme cases of bistochastic matrices. In this paper, we define a scalar quantity called the matching factor on the Birkhoff polytope. Given a bistochastic matrix, we define the matching factor by taking the product of the squares of the Euclidean norms of each row and column and show that permutation matrices and uniform matrices maximize and minimize the matching factor, respectively. We also extend this definition of scalar matching factor to a larger class of matrices and show similar maximization and minimization properties.  
\end{abstract}

\section{Introduction}

Birkhoff polytope is defined as the set of all bistochastic matrices. Bistochastic matrices appear in various fields of study, such as the theory of Markov chains \cite{maksimov1970convergence}\cite{mukhamedov2018stable}, the theory of Majorization \cite{cohen1998comparisons}, and several physical problems \cite{brualdi2006combinatorial}. Random bistochastic matrices have also been considered \cite{cappellini2009random}. Additionally, they are important in the theory of optimization and matching \cite{munkres1957algorithms}\cite{linderman2017reparameterizing}. Permutation matrices are a special class of bistochastic matrices, as they form the building blocks of bistochastic matrices according to the celebrated Birkhoff-Von Neumann theorem \cite{birkhoff1946three}. This theorem states that every bistochastic matrix can be written as a convex combination of permutation matrices. In the literature, there are a few characterizations of permutation matrices \cite{cruse1975proof}. Permutation matrices are also useful in describing perfect matching in bipartite graphs. Another important class of bistochastic matrices is uniform matrices, which appear as the limiting distribution of a Markov chain whose transition matrix is a bistochastic matrix. In the Markov chain literature, permutation transition matrices refer to completely deterministic chains, whereas uniform chains refer to uniform randomization over the state space of the Markov chains.
An interesting question to ask is whether there is a simple scalar quantity (which is easy to compute) that characterizes these two extreme cases. So far, to the best of our knowledge, no such characterization has been proposed. There have been a few works on the characterization of stochastic matrices, see e.g., \cite{VONBELOW20091273}. In this paper, we contribute to the literature by defining a single scalar quantity on the Birkhoff polytope  whose maximum and minimum values characterize permutation and uniform matrices, respectively. We organize the paper as follows: in section \ref{sec2}, we give all the definitions and notations to be used in this paper; in section \ref{sec3}, we state and prove the main results; in section \ref{sec4}, we conclude with certain remarks.

\section{Definitions and Notations} \label{sec2}

\textbf{Notations :} $\mathbb{R}, \mathbb{N}$ will denote the set of all real numbers and natural numbers respectively. We will call a square matrix $A$ of dimension $n \times n$ to be a matrix of order $n$. For a matrix $A$, $A_{k.}$ (note the dot after $k$) will denote the $k$th row of the matrix and $A_{.k}$ (note the dot before $k$) will denote the $k$-th column of the matrix. we call a matrix $A$  to be non-negative if all the entries are non-negative. $\|.\|_{2}^{2}$ will denote the squared euclidean norm.

for $n \in \mathbb{N}$, we define six types of matrices of order $n$ below,
\begin{definition}
    We say a non-negative matrix $B$ of order $n$ to be a bistochastic matrix if
    $$ \sum_{i=1}^{n} B_{ki} = \sum_{j=1}^{n} B_{jk} = 1, \ \ k \in \{1,\dots,n\} $$
\end{definition}
Some examples of bistochastic matrices are 
$$\begin{bmatrix}
    1/2 & 1/2 \\
    1/2   & 1/2
\end{bmatrix} \ , \ \begin{bmatrix}
    1/3 & 1/3 & 1/3 \\
    0 & 2/3  & 1/3 \\
    2/3 & 0 & 1/3
\end{bmatrix}$$

We denote the set of all bistochastic matrices of order $n$ as $\mathcal{B}_{n}$ (Birkhoff polytope)
\begin{definition}
    We say a bistochastic matrix $P$ of order $n$ to be a permutation matrix if each row and column has exactly only one-zero entry and that non-zero entry is equal to one.
\end{definition}

Some examples of permutation matrices are 
$$\begin{bmatrix}
    0 & 1 \\
    1   & 0
\end{bmatrix} \ , \ \begin{bmatrix}
    0 & 1 & 0 \\
    0 & 0  & 1 \\
    1 & 0 & 0
\end{bmatrix}$$

We denote the set of all permutation matrices of order $n$ by $\mathcal{P}_{n}$

\begin{definition}
    We say a bistochastic matrix $U$ of order $n$ to be a uniform matrix if
    $$U_{ij} = \frac{1}{n} ,\ \ \ i,j \in \{1,\dots,n\}$$
\end{definition}

Some examples of uniform matrices are 
$$\begin{bmatrix}
    1/2 & 1/2 \\
    1/2   & 1/2
\end{bmatrix} \ , \ \begin{bmatrix}
    1/3 & 1/3 & 1/3 \\
    1/3 & 1/3  & 1/3 \\
    1/3 & 1/3 & 1/3
\end{bmatrix}$$

Clearly, there is only one uniform matrix of order $n$, we denote that uniform matrix by $U_n$.

Next we state three more definitions that generalizes above three standard definitions

\begin{definition}
    We say a non-negative matrix $A$ to be a $*$-positive matrix if each row and column contain at least one positive entry
\end{definition}

Note that any bistochastic matrix is a $*$- positive matrix and some other examples of $*$-positive matrix are 
$$\begin{bmatrix}
    1/2 & 0 \\
    1/2   & 1/2
\end{bmatrix} \ , \ \begin{bmatrix}
    1/3 & 0 & 0 \\
    1/3 &  7/8 & 0 \\
    1/2 & 0 & 1/3
\end{bmatrix}$$

We denote the set of all $*$-positive matrices of order $n$ as $\mathcal{B}^{*}_{n}$

\begin{definition}
    We say a non-negative matrix $A$ to be a $*$-permutation matrix if each row and column contain exactly one positive entry
\end{definition}

Note that any permutation matrix is a $*$- permutation matrix some other examples of $*$-permutation matrix are 
$$\begin{bmatrix}
    1/2 & 0 \\
    0  & 1/32
\end{bmatrix} \ , \ \begin{bmatrix}
    1/7 & 0 & 0 \\
     0 &  0 & 5 \\
    0 & 1/5 & 0
\end{bmatrix}$$

We denote the set of all $*$-permutation matrices of order $n$ as $\mathcal{P}^{*}_{n}$
\begin{definition}
    We say any $*$- positive matrix $A$ to be a $*$-uniform matrix if all the entries are equal.
\end{definition}

Note that any uniform matrix is a $*$- uniform matrix, some other examples of $*$-uniform matrix are 
$$\begin{bmatrix}
    1/6 & 1/6  \\
    1/6  & 1/6
\end{bmatrix} \ , \ \begin{bmatrix}
    1 & 1 & 1 \\
     1 &  1 & 1 \\
    1 & 1 & 1
\end{bmatrix}$$

We denote the set of all $*$-positive matrices of order $n$ as $\mathcal{U}^{*}_{n}$

\section{Main result} \label{sec3}

Before we state and prove our main characterization theorem, we first state an important observation to be used in the proof.

\begin{lemma}\label{lemma}
    if $a_1,\dots,a_n$ are $n$ non-negative real numbers, then only one of them is equal to one and rest are zero if and only if 
    
    \begin{equation}\label{0.1}
        \sum_{k=1}^{n} a_{k}^2 = (\sum_{k=1}^{n} a_k) ^ 2  = 1
    \end{equation}
\end{lemma}

\begin{proof} Note that \ref{0.1} implies after expanding the square 
$$\sum_{i < j} a_i a_j = 0$$

And the above holds if and only if at most one of them is non-zero, but we also have that $\sum_{k=1}^{n} a_{k}^2 = 1$, and all $a_1,\dots,a_n$ are non-negative, hence exactly one of them is equal to one and rest are zero. Other direction is trivial.
\end{proof}

Next we define some quantities for stating our main result. For a bistochastic matrix $B$ of order $n$  let $B_{1.}, \dots B_{n.}$ denote the rows of $B$ and let $B_{.1}, \dots B_{.n}$ denote the columns of $B$. For each $k$-th row and $k$-column, we define \textit{$k$-th Matching factor} of $B$ as follows

$$\lambda(k) = \|B_{k.}\|_{2}^{2} \|B_{.k}\|_{2}^{2} \ \ \ \ k = 1, \dots, n$$ 

Intuitively if $\lambda(k)$ is near one, then one of the entry is nearly one in each row and column and rest are near zero, where as if $\lambda(k)$ in small, then many entries can be non-zero. The reason behind this intuition is the geometry of the $l_2$ ball and $l_1$ ball in $\mathbb{R}^n$, see \cite{tibshirani1996regression}. we gave the name Matching factor because the corresponding row and column matches to exactly one non-zero entry when $\lambda$ is equal to one.

For a bistochastic matrix $B$, we denote the \textit{Matching factor} of $B$ as $M(B)$ and  define it as follows

$$M(B) = \prod_{k=1}^{n} \lambda(k)$$

Intuitively the quantity $M(B)$ captures the entire matching factor of the matrix by combined contribution of the matching factor of each row and column. Next we state our characterization theorem.

\begin{theorem}\label{characterization theorem}
 for any bistochastic matrix $B$ of order $n$, the matching factor $M(B)$ satisfies the inequalities

 $$\frac{1}{n^{2n}} \leq M(B) \leq 1$$ 

 and the lower bound and upper bounds are achieved 
 \begin{equation}
     \max_{B \in \mathcal{B}_n} M(B) = 1  \ \ \ \argmax M(B) = \mathcal{P}_n
 \end{equation} 
  and 
  \begin{equation}
     \min_{B \in \mathcal{B}_n} M(B) = \frac{1}{n^{2n}}  \ \ \ \argmin M(B) = \{\mathcal{U}_n\}
 \end{equation} 
\end{theorem}

\begin{proof}

First we prove 3.2. Note that if $B$ is permutation matrix, then by definition each row and column of $B$ will have exactly non-zero entry which is one and rest zero, hence we will get that

$$\|B_{.k}\|_{2}^2 = \|B_{k.}\|_{2}^2 =  1\ \ \ \forall \ \ k \in \{1,\dots,n\}$$.

Hence we get $M(B) = 1$.

Now to prove the other direction. suppose we have $M(B) = 1$.

By now non-negativity of each entries of $B$, and by the sum constraints we must have

\begin{align}
\|B_{.k}\|_{2}^2 = \sum_{i=1}^{n} B_{ik}^2  & \leq (\sum_{i=1}^{n} B_{ik})^2 = 1 \ \ \ \forall \ \ k \in \{1,\dots,n\} \\
\|B_{k.}\|_{2}^2 = \sum_{i=1}^{n} B_{ki}^2  & \leq (\sum_{i=1}^{n} B_{ki})^2 = 1 \ \ \ \forall \ \ k \in \{1,\dots,n\} \\
\end{align} 

Hence for general bistochastic matrix we must have that $\lambda(k) \leq 1 \ \forall \ k  \in  \{1,\dots,n\} $, which implies $M(B) \leq 1$. 

Now from our assumption we have $M(B) = 1$. Hence we must have $\lambda(k) = 1 \ \forall \ k \in \ \ \{1,\dots,n\}$, which in turn implies that all inequalites in 3.4 and 3.5 must be equalities. Hence now by using lemma \ref{lemma}, we get each row and each column of $B$ must have exactly non-zero entry equal to one and rest zero which in turn implies that $B$ must be a permutation matrix, hence we are done. 

Now we prove 3.3. The proof follows from the well known Cauchy-Schwarz inequality. 

Note that by Cauchy-Schwarz inequality

\begin{align}
  \|B_{.k}\|^2  =  \sum_{i=1}^{n} B_{ik}^2 & \geq \frac{1}{n} (\sum_{i=1}^{n} B_{ik} )^2 = \frac{1}{n} \ \ \forall \ k  \in \ \ \{1,\dots,n\} \\ 
\|B_{k.}\|^2 =  \sum_{j=1}^{n} B_{kj}^2 & \geq \frac{1}{n} (\sum_{j=1}^{n} B_{jk} )^2 = \frac{1}{n} \ \ \forall \ k  \in \ \ \{1,\dots,n\}
\end{align}

Hence from 3.7 and 3.8, we will get that $\lambda(k) \geq \frac{1}{n^2} \ \ \forall k \in \{1,\dots,n\}$

Hence we must get that for any bistochastic matrix,  $M(B) = \prod_{k=1}^{n} \lambda(k) \geq \frac{1}{n^{2n}}$

and note that the equality $M(B) = \frac{1}{n^{2n}}$ holds if and only if all the inequalities in 3.7 and 3.8 are equalities and by the conditions of equality in cauchzy schwarz inequality, we must have that all entries should be equal and since each row and column sum upto one, we must have $B_{ij} = \frac{1}{n} \ \ \forall \ i,j \in \{1,\dots,n\}.$. Hence the minimum is achieved only at the uniform matrix. Hence we have established our result. 

\end{proof}

Next we state and prove similar results involving $*$- positive matrix, $*$- permutation matrix and $*$-uniform matrix.

 we again define similar quanties for $*$- positive matrix. For a $*$- positive matrix $A$ of order $n$  let $A_{1.}, \dots A_{n.}$ denote the rows of $A$ and let $A_{.1}, \dots A_{.n}$ denote the columns of $A$. For each $k$-th row and $k$-th column, we define \textit{$k$-th Matching factor} of $A$ as follows

$$\lambda^{*}(k) = \frac{\|A_{k.}\|_{2}^{2} \|A_{.k}\|_{2}^{2}}{(\sum_{i=1}^{n} A_{ki})^2 (\sum_{j=1}^{n} A_{jk})^2}  \ \ \ \ k = 1, \dots, n$$ 

Intuitively if $\lambda(k)$ is near one, then one of the entry is nearly one and rest are near zero, where as if $\lambda(k)$ in small, then many entries can be non-zero. 

For a $*$-positive matrix, we denote the \textit{Matching factor} of $A$ as $M^{*}(A)$ and  define it as follows

$$M^{*}(A) = \prod_{k=1}^{n} \lambda^{*}(k)$$

Note that above definitions of Matching factor reduces to the same definition of Matching factor for bistochastic matrices. Next we state our characterization theorem for $*$-positive matrices.

\begin{theorem}
 for any $*$-positive matrix $A$ of order $n$, the matching factor $M^{*}(A)$ satisfies the inequalities

 $$\frac{1}{n^{2n}} \leq M^{*}(A) \leq 1$$ 

 and the lower bound and upper bounds are achieved 
 \begin{equation}\label{upper bound}
     \max_{A \in \mathcal{A}^{*}_n} M^{*}(A) = 1  \ \ \ \argmax M(A) = \mathcal{P}^{*}_n
 \end{equation} 
  and 
  \begin{equation}\label{lower bound}
     \min_{A \in \mathcal{A}^{*}_n} M^{*}(A) = \frac{1}{n^{2n}}  \ \ \ \argmin M^{*}(A) = \mathcal{U}^{*}_n\
 \end{equation} 
\end{theorem}

\begin{proof}

First we prove 3.9. Note that if $A$ is $*$-permutation matrix, then by definition each row and column of $A$ will have exactly non-zero entry and rest zero, hence we will get that

\begin{align*}
    \|A_{.k}\|_{2}^2 = \sum_{i=1}^{n} A_{ik}^2 & = (\sum_{i=1}^{n} A_{ik})^2  \ \ \forall \ \ k \in \{1,\dots,n\} \\
    \|A_{k.}\|_{2}^2 =  \sum_{i=1}^{n} A_{ki}^2  & = (\sum_{i=1}^{n} A_{ki})^2 \ \ \forall \ \ k \in \{1,\dots,n\}
\end{align*}
Hence by combining, we get 
$$ \lambda^{*}(k) = \frac{\|A_{.k}\|_{2}^2 \|A_{k.}\|_{2}^2}{(\sum_{i=1}^{n} A_{ki})^2 (\sum_{j=1}^{n} A_{jk})^2} =  1\ \ \ \forall \ \ k \in \{1,\dots,n\}$$.

Hence we get $M^{*}(A) = 1$.

Now to prove the other direction. suppose we have $M^{*}(A) = 1$.

By non-negativity of each entries of $A$, and that each row and column has atleast one positive entry, and from the trivial observation that for any non-negative numbers $a_1, \dots, a_n$ we must have $\sum_{k=1}^{n} a^{2}_k \leq (\sum_{k=1}^{n} a_{k})^2$, so we get

\begin{equation}
    \frac{\|A_{.k}\|_{2}^2 \|A_{k.}\|_{2}^2}{(\sum_{i=1}^{n} A_{ki})^2 (\sum_{j=1}^{n} A_{jk})^2} \leq  1\ \ \ \forall \ \ k \in \{1,\dots,n\}
\end{equation}

Hence for any $*$- positive matrix, we must have that $\lambda^{*}(k) \leq 1 \ \forall k \  \in  \{1,\dots,n\} $, which implies $M^{*}(A) \leq 1$. 

Now from our assumption we have $M^{*}(A) = 1$. Hence we must have $\lambda^{*}(k) = 1 \ \forall k \in \ \ \{1,\dots,n\}$ Which in turn implies that all inequalites in 3.11 must be equalities, So we will have
\begin{align*}
     \sum_{i=1}^{n} A_{ik}^2 & = (\sum_{i=1}^{n} A_{ik})^2  \ \ \forall \ \ k \in \{1,\dots,n\} \\
     \sum_{i=1}^{n} A_{ki}^2  & = (\sum_{i=1}^{n} A_{ki})^2 \ \ \forall \ \ k \in \{1,\dots,n\}
\end{align*}

Now we state a slight modification of lemma \ref{lemma} below

\textbf{Modification of lemma \ref{lemma}:} 
if $a_1,\dots,a_n$ are $n$ non-negative real numbers, then only exactly one of them is non-zero if and only if 
    
\begin{equation}\label{0.2}
        \sum_{k=1}^{n} a_{k}^2 = (\sum_{k=1}^{n} a_k) ^ 2  > 0
\end{equation}

The proof of the above modification is a trivial modification of the proof of lemma \ref{lemma}. So we get each row and each column of $A$ must have exactly one non-zero entry  and rest are zero which in turn implies that $A$ must be a $*$- permutation matrix, hence we are done. 

Now we prove 3.10. The proof again follows from  Cauchy-Schwarz inequality. 

Note that by Cauchy-Schwarz inequality

\begin{equation}
    \frac{\|A_{.k}\|_{2}^2 \|A_{k.}\|_{2}^2}{(\sum_{i=1}^{n} A_{ki})^2 (\sum_{j=1}^{n} A_{jk})^2} \geq \frac{1}{n^2}\ \ \ \forall \ \ k \in \{1,\dots,n\}
\end{equation}
Hence from 3.13, we will get that $\lambda^{*}(k) \geq \frac{1}{n^2} \ \ \forall k \in \{1,\dots,n\}$

Hence we must get that for any $*$- positive matrix,  $M^{*}(A) = \prod_{k=1}^{n} \lambda^{*}(k) \geq \frac{1}{n^{2n}}$

and note that the equality $M^{*}(A) = \frac{1}{n^{2n}}$ holds if and only if all the inequalities in 3.13 are equalities and by the conditions of equality in cauchzy schwarz inequality, we must have that all entries should be equal Hence the minimum is achieved only for the $*$-uniform matrix. Hence we have established our result. 

\end{proof}

\section{Conclusion and remarks} \label{sec4}
We defined a scalar matching factor on a bistochastic matrix and showed that the lower and upper bounds characterize the permutation matrices and uniform matrices. We also defined similar quantities for a larger class of matrices, known as $*$-positive matrices, and demonstrated similar maximization and minimization properties. The scalar matching factor on bistochastic matrices can be of independent interest in optimization, matching, and the theory of Markov chains. One possible application could be checking or validating whether a large bistochastic matrix is a permutation matrix by simply computing our matching factor rather than checking row by row. This matching quantity seems to provide a notion of how far any bistochastic matrix is from being a permutation matrix, though more thought will be needed to correctly formulate such notions.
\section{Acknowledgements}

The Author acknowledge the summer support from the Department of Statistics and Data science at Cornell university. Special thanks to Suraj Dash of UT Austin for taking part in discussion of some of the ideas presented in the paper.
\bibliographystyle{acm}
\bibliography{sample}

\end{document}